\documentclass{amsart}

\usepackage{amsmath,amssymb,amsfonts,amsthm}
\usepackage{mathtools}
\usepackage{mathrsfs}

\usepackage{pgf,tikz}
\usepackage{mathrsfs}
\usepackage{pgf,xcolor}
\usetikzlibrary{arrows}
\usepackage[all]{xy}
\usepackage{todonotes}

 \DeclareMathOperator{\tr}{tr}

\DeclareMathOperator{\Span}{span}
\DeclareMathOperator{\Hom}{Hom}
\DeclareMathOperator{\U}{U}

\newcommand{\CP}{{\mathbb C}P}
\newcommand{\C}{{\mathbb C}}
\newtheorem{theorem}{Theorem}[section]
\newtheorem{corollary}[theorem]{Corollary}
\newtheorem{lemma}[theorem]{Lemma}

\newtheorem{definition}[theorem]{Definition}
\newtheorem{proposition/definition}[theorem]{Proposition/Definition}

\theoremstyle{definition}

\newtheorem{remark}[theorem]{Remark}

\usepackage{pdfpages}
\usepackage{hyperref}

\title[Constantly curved minimal two-spheres in unitary groups]
{Constantly curved minimal immersions of the two-sphere in unitary groups}
\author[R. Pacheco]{Rui Pacheco}
\address{Centro de Matem\'{a}tica e Aplica\c{c}{\~{o}}es (CMA-UBI), Universidade da Beira Interior, 6201 -- 001
	Covilh{\~{a}}, Portugal.}
\email{rpacheco@ubi.pt, mehmood.ur.rehman@ubi.pt}

\author[M. U. Rehman]{Mehmood Ur Rehman}

\thanks{The first author was partially supported by Funda\c{c}\~{a}o para a Ci\^{e}ncia e Tecnologia through the project UIDB/00212/2025. The second author was partially supported by Funda\c{c}\~{a}o para a Ci\^{e}ncia e Tecnologia through the grant UI/BD/153058/2022.}


\begin{document}

	\maketitle

\begin{abstract}In this article,
we investigate rigidity results for  constantly curved minimal immersions of the two-sphere $S^2$  into the unitary group $\U(n)$. Using loop group methods for harmonic maps, we establish a correspondence between such immersions and a distinguished class of constantly curved holomorphic immersions of $S^2$ into finite-dimensional Grassmannians. 
In the case $\U(3)$, we classify the constantly curved minimal immersions of $S^2$ with uniton number one and prove that, under a natural unramifiedness condition, those of uniton number two are $S^1$-invariant; as a consequence, every constantly curved totally unramified minimal immersion $S^2\to \U(3)$ of uniton number two is unitarily congruent to the composition of the first Gauss map of the Veronese curve in $\mathbb{CP}^2$ with the Cartan embedding $\mathbb{CP}^2\hookrightarrow \U(3)$.

\end{abstract}

\maketitle

\section{Introduction}
We investigate rigidity results for  minimal immersions with constant curvature of the two-sphere $S^2$ into the unitary group $\U(n)$. This falls within an important topic in differential geometry, going back to the work of O. Borůvka, where, for each $m>0$, a   linearly full minimal immersion of $S^2$ into $S^{2m}$  with positive constant curvature, now known as the Veronese-Borůvka sphere, was constructed (see \cite{kenmotsu1997} and references therein). 
Later, E. Calabi \cite{calabi} proved that, up to rigid motions, these are the only compact minimal surfaces with constant positive curvature in $S^n$, for any $n$.  The  classification of minimal immersions of surfaces with constant curvature in $S^n$ was completed by N. Wallach \cite{wallach}, K. Kenmotsu \cite{kenmotsu1976} and R. Bryant \cite{bryant1985}.    

Several researchers have investigated the analogous problem for minimal surfaces in other Riemannian symmetric spaces. E. Calabi \cite{Calabi1958} (see also \cite{Lawson}) established a foundational result by showing that any linearly full holomorphic curve with constant curvature
of $S^2$ into the complex projective space $\C P^n$ is the Veronese curve $V_0^n:S^2 \rightarrow \mathbb{C}P^n$, up to
unitary congruence. J. Bolton et al. \cite{Bolton} proved that any linearly full  minimal immersion with constant curvature of  $S^2$ into $\CP^n$ belongs to the  \emph{Veronese sequence} 	$V_0^n,\ldots, V^n_n :S^2 \rightarrow \mathbb{C}P^n$.  The local version of this result was found by Q.-S. Chi and Y. Zheng \cite{Chi-Zheng}. Z. Li and Z.-H. Yu \cite{Li99} provided a complete classification  of  minimal immersions with constant curvature of $S^2$  into $G_2(\C^4)$.  For other Grassmannians, however, a complete classification of such immersions appears to be very difficult to obtain, and only partial results are currently known.  In \cite{HeJiao2014} and \cite{HeJiao2015}, respectively, the authors 
classified linearly full \emph{totally unramified} conformal minimal
immersions with constant curvature of $S^2$ into  the quaternionic projective space $\mathbb{H}P^2$ and $G_2(\C^5)$, respectively.

Our approach to the $\U(n)$ case relies on the loop group methods for harmonic maps developed by K. Uhlenbeck \cite{uhlenbeck} and G. Segal \cite{segal}. K. Uhlenbeck was the first to observe a correspondence between harmonic maps from a simply connected Riemann surface into $\U(n)$ and \emph{extended solutions} into the loop group $\Omega \U(n)$, from which the original maps can be recovered \cite[Theorem 2.2]{uhlenbeck}. 
When the domain is $S^2$, every nonconstant harmonic map
$\psi:S^2\to \U(n)$ is a branched minimal immersion and has finite
uniton number, i.e., it admits a \emph{polynomial} extended solution  \cite[Theorem 11.5]{uhlenbeck}.  G. Segal \cite{segal} showed that the Grassmannian model of loop groups provides a natural framework for these results. Within the Grassmannian model, the polynomial extended solution corresponds to a holomorphic map $X$ from $S^2$ into a certain finite-dimensional Grassmannian. We will say that the original map $\psi$ is \emph{totally unramified}  if the corresponding holomorphic $X$ is totally unramified (in the sense of \cite{HeJiao2014,HeJiao2015}). 
In Section \ref{sec:finitegrass}, Lemma \ref{metricslemma}, we prove that $\psi$ is a minimal immersion of $S^2$ with constant curvature  if and only if $X$ is a minimal immersion with constant curvature of $S^2$.  
This reduces our classification problem to that of classifying a distinguished class of holomorphic immersions with constant curvature into finite-dimensional Grassmannians. To address this reduced problem, we make essential use of the theory of harmonic sequences and harmonic diagrams developed in \cite{burstall-wood} (see also \cite{PW1}).

The first nontrivial case is $n=3$. 
By \cite{segal,uhlenbeck}, every harmonic map $\psi:S^2\to\U(3)$ has \emph{uniton number} $r\leq 2$; that is, it  admits a polynomial extended solution of degree $r\leq 2$. 
The case $r=1$ is readily handled (Theorem \ref{uniton1}), whereas the case $r=2$ is considerably more challenging. In Theorem \ref{thm:unramified2}, we prove that, if $\psi:S^2\to\U(3)$ is a minimal immersion of uniton number $r=2$ with constant curvature satisfying some natural unramifiedness condition (in particular, this includes the case where $\psi$ is totally unramified), then $\psi$ is \emph{$S^1$-invariant} (in the sense of \cite[\S 3]{burstallguest}). Combining this with the results of Section \ref{sec:s1invariant} on $S^1$-invariant minimal immersions with constant curvature, we conclude that every such  minimal immersion is unitarily congruent to the composition  of
the first Gauss map $V_1^2$ of the Veronese curve $V_0^2$
 with the Cartan embedding $\C P^2\hookrightarrow \U(3)$.

\section{Preliminaries}

Let $\U(n)$ denote the unitary group of degree $n$, with identity element $I_n$,  equipped with the bi-invariant metric
$$h_n(\xi,\eta)=\frac18\tr (\xi \eta^*),\quad \xi,\eta\in \mathfrak{u}(n),$$
where $\mathfrak{u}(n)$ denotes the Lie algebra of $\U(n)$ and $*$ denotes the Hermitian adjoint. The scaling factor $\frac18$ agrees  with the convention in \cite{JiaoPeng2003}.  
Let $M$ be a Riemann surface and $\psi:M \to \U(n)$ be a smooth map. Consider the $\mathfrak{u}(n)$-valued $1$-form  
$\tfrac{1}{2}\psi^{-1}d\psi=A^{\psi}_{z}dz+A^{\psi}_{\bar z}d \bar z$, with $(A^{\psi}_{\bar z})^*=-(A^{\psi}_{z} )$,  
where $z$ is a  local complex coordinate on $M$. If $\psi$ is conformal, then the pullback of $h_n$ by $\psi$ is  locally given by
\begin{equation}\label{pullbackpsi}
\psi^*h_n=\tr  A^\psi_z(A^\psi_z)^*dzd\bar z.
\end{equation}

Recall from \cite[\S 1]{uhlenbeck}  that $\psi$ is \emph{harmonic}  if and only if 
$(A^{\psi}_{z})_{\bar{z}}+(A^{\psi}_{\bar z})_{{z}}=0.$ When $M=S^2$, every harmonic map $\psi:S^2\to \U(n)$ is weakly conformal, since its Hopf differential is holomorphic and every holomorphic quadratic differential on $S^2$ vanishes identically. Thus, every nonconstant harmonic map $S^2\to \U(n)$ is a branched minimal immersion. 

Assume that $M$ is simply connected. By \cite[Theorem 2.2]{uhlenbeck}, every harmonic map  $\psi:M \to \U(n)$ admits an \emph{extended solution} $$\Psi:M\to \Omega \U(n)=\{\gamma:S^1\to\U(n) \,\, \mathrm{smooth}\,|\, \gamma(1)=I_n  \},$$ 
which is unique up to left multiplication by a constant loop  and satisfies
\begin{equation}\label{PSI}
\Psi^{-1}d\Psi=(1-\lambda^{-1})A^\psi_zdz+(1-\lambda)A^\psi_{\bar z}d\bar z,\quad \lambda\in S^1.
\end{equation}
We set $\Psi_\lambda (z)=\Psi(z)(\lambda)$. The harmonic map $\psi$ is recovered from its extended solution $\Psi$  by 
the formula $\psi=\Psi_{-1}$, up to left multiplication by a constant element of $\U(n)$.

Next we consider 
the  \emph{Grassmannian model} \cite{pressley,segal} (see also \cite{APW1})  for based loop groups. 
In this model,  a loop $\gamma\in \Omega \U(n)$ corresponds to the closed subspace $W=\gamma \mathcal{H}_+$ of $L^2(S^1,\C^n)$, the Hilbert space  of square-integrable $\C^n$-valued functions on $S^1$, where  $\mathcal{H}_+$  denotes the closed subspace of $L^2(S^1,\C^n)$  consisting of Fourier series whose negative Fourier coefficients vanish. 
The subspace $W$ is \emph{shift-invariant}, meaning it is closed under multiplication by $\lambda \in S^1$: 
\begin{equation}\label{shiftinvariance}
	\lambda W \subset W.
\end{equation}
Accordingly, an extended solution $\Psi$ corresponds to a smooth vector subbundle  $W$ of $\underline{\mathcal{H}}:=M\times L^2(S^1,\C^n)$, whose fiber at $z$ is the shift-invariant subspace given by
$W(z)=\Psi(z) \mathcal{H}_+$. 
The harmonicity condition $(A^{\psi}_{z})_{\bar{z}}+(A^{\psi}_{\bar z})_{{z}}=0$ leads to the following properties of $W$:
\begin{align}\label{lambda1}
	 \partial_{ \bar z}\Gamma(W)&\subset \Gamma(W); \\ \lambda\partial_z\Gamma(W) &\subset \Gamma(W); \label{lambda2} \end{align}
where $\Gamma(W)$ is the space of all smooth sections of $W$, and $\partial_z,\partial_{ \bar z}$ are the derivatives with respect  to $z,\bar z$. The property \eqref{lambda1} means  that $W$ is a holomorphic vector subbundle of $\underline{\mathcal{H}}$.

When the domain is the two-sphere $S^2$, K. Uhlenbeck \cite[Theorem 11.5]{uhlenbeck} observed that any harmonic map $\psi : S^2 \to \U(n)$  has \emph{finite uniton number}, that is, it admits a \emph{polynomial} extended solution, that is, an extended solution of the form
$\Psi= \sum_{i=0}^{r}\lambda^{i} A_i$, with $A_i: S^2 \to \mathfrak{u(n)} \otimes \C$ smooth,  for some nonnegative integer $r$.  The corresponding $W=\Psi \mathcal{H}_+$ satisfies the following inclusions 
\begin{equation}\label{unitonr}
\lambda^r \mathcal{H}_+ \subset W \subset \mathcal{H}_+.
\end{equation}
The smallest such $r$ is called the uniton number of $\psi$. From \cite[Proposition 2.12]{segal}, we have $r<n$; moreover, in view of \cite[Corollary 2.10]{segal}, we can assume, without loss of generality, that 
 \begin{equation}\label{span}
	\mbox{	$\bigcup_{z \in S^2}W(z)$ spans $\mathcal{H}_+$}.  
\end{equation}  
If $W=\Psi \mathcal{H}_+$ satisfies \eqref{span}, we say that $\Psi$ (or $W$) is a \emph{normalized} extended solution. 

\subsection{Harmonic maps into complex Grassmannians}
Let $G_k(\C^{n})$ denote the Grassmannian of all $k$-dimensional complex subspaces of $\C^{n}$. We equip $G_k(\C^{n})$  with
the invariant metric $h_{n}$
 induced from the bi-invariant metric on $\U(n)$ by the Cartan embedding  $\iota:G_k(\C^{n}) \hookrightarrow \U(n)$, defined by 
$\iota(X)=\pi_X-\pi_X^\perp$, where  $\pi_X$ and $\pi_X^\perp$ denote the  orthogonal projections, with respect to the standard Hermitian inner product on $\C^{n}$, onto $X$ and $X^\perp$, respectively. Given a smooth map $X:M\to G_k(\C^{n})$ (equivalently, a smooth rank-$k$ complex vector subbundle of $M\times \C^n$), we denote  
 its  \emph{$\partial_z$-second fundamental form} \cite[\S  1]{burstall-wood}  by $A'_X\in\Hom(X,X^\perp)$, which is locally defined by
$A'_X(s)=\pi_X^\perp \partial_zs$, with $s\in\Gamma(X)$.  In what follows, we will use the  well-established theory  \cite{burstall-wood,chern-wolfson,PW1} of \emph{harmonic sequences} and \emph{diagrams} associated with harmonic maps from a Riemann surface into complex Grassmannians. Given any such harmonic map $X$, we denote its \emph{$i$th $\partial_z$-Gauss bundle}  by $G^{(i)}(X)$, for $i\in \{1,2,\ldots\}$.

If $X:S^2\to  G_k(\C^{n})$ is holomorphic, then we can take the smallest $r$ such that  $G^{(r+1)}(X)=\{0\}$, and its \emph{fundamental diagram} is given by 
\begin{equation}\label{fundamentaldiagramgeneral}
	\begin{gathered}
		{\xymatrixrowsep{1.5pc}{\xymatrix{
					X\ar[r]^{\hspace{-0.2in} A'_X}  & G^{(1)}(X)\ar[r]^{A'_{G^{(1)}(X)}}   & G^{(2)}(X) \ar[r]^{A'_{G^{(2)}(X)}} &\ldots  \ar[r]^{\!\!\!\!\!\!\!\!\!\!\!\!\!A'_{G^{(r-1)}(X)}} & G^{(r)}(X) \ }}}
	\end{gathered}.
	\underline{}\end{equation}   
Recall from \cite[Definition 2.2]{HeJiao2014} that such holomorphic map $X$ is said to be \emph{totally unramified}  if each horizontal arrow in \eqref{fundamentaldiagramgeneral} is  \emph{unramified}, that is, 
\begin{equation}\label{tudef}
\det A'_{G^{(i)}(X)} (A'_{G^{(i)}(X)})^*\neq 0\,\, \mbox{everywhere on $S^2$, for all $i\in\{0,\ldots,r-1\}$,}\end{equation} 
where $G^{(0)}(X)=X$. 
\section{Reduction to constantly curved holomorphic immersions in $G_k(\C^{rn})$}\label{sec:finitegrass}

 Given a  harmonic map $\psi: S^2\to \U(n)$ of uniton number $r>0$, we can construct a holomorphic map from $S^2$ into a finite-dimensional complex Grassmannian as follows:

Let  $W=\Psi\mathcal{H}_+$ be a normalized extended solution of $\psi$, hence satisfying conditions \eqref{shiftinvariance}--\eqref{span}.  
Set
\begin{equation}\label{Vdef}
X=W\cap \bigoplus_{j=0}^{r-1}\lambda^j\C^n, \quad \C^{rn}\cong \bigoplus_{j=0}^{r-1}\lambda^j\C^n.
\end{equation} Due to the holomorphicity of $W$,  $X$ has constant rank $k< rn$ away from a finite subset of $S^2$. We can fill out the zeros to obtain a globally defined holomorphic   map $X:S^2\to G_k(\C^{rn})$, for some positive integer $k$.
 Then $W$ admits the following orthogonal decomposition:
\begin{equation}\label{grassmanian model}
	W= X \oplus \lambda^r\mathcal{H}_+.
\end{equation}
Condition \eqref{span} implies that $X$ is linearly full.   
Moreover, since $X$ is holomorphic, the shift-invariance property \eqref{shiftinvariance} and property  \eqref{lambda2} are respectively equivalent to:
\begin{equation}\label{propV}
	a)	\,\,  \lambda X\subset X\mod\lambda^r\mathcal{H}_+; \qquad b)\,\, 
	\lambda \partial_z \Gamma(X) \subset \Gamma(X) \mod\lambda^r\mathcal{H}_+.  
\end{equation}

Since   $X:S^2\to G_k(\C^{rn})$ is holomorphic, the pullback of the invariant metric $h_{rn}$ is locally given by \cite[equation (15)]{JiaoPeng2003} 
  \begin{equation}\label{pullback}X^*h_{rn}=\tr A'_X(A'_X)^* dzd \bar z.
  	\end{equation} 
 We have the following: 
\begin{lemma}\label{metricslemma}
	Let $\psi:S^2\to \U(n)$ be a harmonic map with uniton number $r>0$, let $W=\Psi\mathcal{H}_+$ be a normalized extended solution of $\psi$, and let  $X:S^2\to G_k(\C^{rn})$ be  the corresponding holomorphic map,  defined by \eqref{Vdef}. Then $\psi$ is an immersion  if and only if $X$ is an immersion. Moreover, in this case, the  induced metrics coincide: $X^*h_{rn}= \psi^*h_n $. 
\end{lemma}
\begin{proof}
	 Let $e_1,\ldots ,e_n$ be the canonical complex  basis of $\C^n$. Since the action of $\Omega \U(n)$ on $\mathcal{H}$ preserves the $L^2$-inner product,    the sections $\Psi(\lambda^i e_j)$,  with $i \in \{0,1,\ldots \}$ and $j \in \{1, \ldots, n\}$,  form a unitary frame for $W$; while the sections $\Psi(\lambda^{-l }e_j)$,  with $l \in  \{1,2,\ldots \}$ and $j\in \{1, \ldots, n\}$,  form a unitary frame for $W^\perp$.  Let us compute the $\partial_z$-second fundamental form ${A}^{\prime}_W$ of $W$ (with respect to the $L^2$ inner product $\langle\cdot ,\cdot \rangle_{L^2}$ on $\mathcal{H}$). 	By \eqref{PSI}, we have $\Psi^{-1}\partial_z\Psi=(1-\lambda^{-1})A^{\psi}_z$, where $A^{\psi}_z=\frac12 \psi^{-1}\partial_z \psi$. 
Then,	for $s=\Psi(\lambda^i e_j)$, we have
	\begin{equation*}\label{secondfundamentalformW}
		\begin{aligned}	{A}^{\prime}_W(s)&= \pi_{W}^\perp  {\partial_z } s
=\pi_{W}^\perp\Psi(1-\lambda^{-1})A^{\psi}_z(\lambda^ie_j)= \Psi \pi_{{\mathcal H}_+}^\perp(1-\lambda^{-1})A^{\psi}_z(\lambda^ie_j).
		\end{aligned}
	\end{equation*}
	Hence,
	\begin{equation*}
	{A}'_W(s) = 
		\begin{cases}
			-\lambda^{-1} \Psi A^{\psi}_z(e_j) & \text{if } i = 0 \\
			0 & \text{if } i \geq 1
		\end{cases}.
	\end{equation*}
	From this, we obtain
	\begin{equation}\label{AW}
		\left\langle {A}'_W(\Psi(\lambda^i e_j)), \Psi(\lambda^{-l}e_m) \right\rangle_{L^2}= 
		\begin{cases}
			-\left \langle A^{\psi}_z(e_j), e_m\right \rangle & \text{if $i = 0$ and $l=1$}  \\
			0 & \text{if $i > 0$ or $l>1$} 
		\end{cases}.
	\end{equation}
Now, let  $X:S^2\to G_k(\C^{rn})$  be defined as in \eqref{Vdef}. For all $s\in \Gamma(X)\subset \Gamma(W)$, since $\partial_zs\in\bigoplus_{j=0}^{r-1}\lambda^j\mathbb C^n$, which is orthogonal to $\lambda^rH_+$, the orthogonal projections $\pi_X^\perp$ and $\pi_W^\perp$ agree on $\partial_zs$. Hence, for all $s\in\Gamma(X)\subset\Gamma(W)$, we have
	\begin{equation}\label{AVAW}
		{A}^{\prime}_X(s)= \pi_{X}^\perp\circ {\partial_z } s= \pi_{W}^\perp\circ {\partial_z } s={A}^{\prime}_W(s); 
\end{equation}
while ${A}'_W(s)=0$ for every $s\in \Gamma(\lambda^r{\mathcal H}_+)$. 
	Placing  \eqref{AW} and \eqref{AVAW} together gives, in view of \eqref{pullbackpsi} and \eqref{pullback},
	\begin{equation*}
		\begin{aligned}
			X^*h_{rn}&=	 \tr {A}'_X({A}'_X)^*dzd\bar z =  \tr {A}'_W({A}'_W)^*dzd\bar z= \tr A^{\psi}_z( A^{\psi}_z)^*dzd\bar z=\psi^*h_n,
		\end{aligned}
	\end{equation*}
	and the result follows.
\end{proof}

Lemma \ref{metricslemma} therefore reduces the classification of minimal immersions $\psi:S^2\to \U(n)$ with constant curvature  to that of  linearly full holomorphic  immersions of constant curvature $X:S^2\to G_k(\C^{rn})$ satisfying  \eqref{propV}.

\begin{definition}
		A harmonic map $\psi:S^2 \to \U(n)$ is said to be \emph{totally unramified} if the associated holomorphic map $X:S^2 \to G_k(\C^{rn})$ defined by \eqref{Vdef}  is totally unramified.
\end{definition}

\section{$S^1$-invariant constantly curved minimal immersions of $S^2$ into $\U(n)$}\label{sec:s1invariant}
An extended solution $\Psi:S^2\to \Omega \U(n)$ is said to be $S^1$-\emph{invariant} \cite[\S 3]{burstallguest} if it takes values in the $\U(n)$-conjugacy class of some homomorphism $\gamma:S^1\to\U(n)$; equivalently, $W=\Psi \mathcal{H}_+$ is of the form 
\begin{equation}\label{s1invariant}
W=X\oplus\lambda^r\mathcal{H}_+,\quad \mbox{with $X= X_0\oplus\lambda X_1\oplus\lambda^2 X_2\oplus\ldots \oplus\lambda^{r-1} X_{r-1}$},
\end{equation}
for some  integer $r>0$, where each $X_j:S^2\to G_{k_j}(\C^n)$ is holomorphic. In this case,  conditions \eqref{propV} reduce to
$$a)\,\, X_j\subset X_{j+1};\qquad b)\,\, \partial_z \Gamma(X_j)\subset \Gamma(X_{j+1});$$ 
for every $j\in\{0,\ldots, r-2\}$. We have
\begin{equation}\label{PSIinvariant}
\Psi_\lambda=\pi_{X_0}+\lambda \pi_{X_1\cap X_0^\perp}+\lambda^2 \pi_{X_2\cap X_1^\perp}+\ldots +\lambda^r\pi_{X_{r-1}^\perp}.
\end{equation}
If $\psi:S^2\to \U(n)$ admits an $S^1$-invariant extended solution, we also say that $\psi$ is \emph{$S^1$-invariant}.
 \begin{theorem}\label{thm:s1invariant}
Let $\Psi:S^2\to \Omega \U(n)$ be a normalized  $S^1$-invariant extended solution. Then $\psi=\Psi_{-1}:S^2\to \U(n)$ is a minimal immersion  with constant curvature  if and only if each $X_j:S^2\to G_{k_j}(\C^n)$ in \eqref{s1invariant} is a linearly full holomorphic immersion with constant curvature. 	
 \end{theorem}
\begin{proof}
Assuming that the $S^1$-invariant extended solution $W=\Psi \mathcal{H}_+=X\oplus \lambda^r\mathcal{H}_+$ is normalized, it is obvious that each  holomorphic map $X_j:S^2\to G_{k_j}(\C^n)$ in \eqref{s1invariant} must be linearly full. 
 
 We have
 \begin{equation}\label{flagmetric}
X^*h_{rn}=\sum_{j=0}^{r-1} X^*_jh_n.  
	 \end{equation}
Assume that $\psi=\Psi_{-1}:S^2\to \U(n)$ is a minimal immersion with constant curvature. By Lemma \ref{metricslemma}, the holomorphic map  $X:S^2\to G_k(\C^{rn})$ is an immersion of constant curvature $K$. From \cite[Theorem 4.1]{JiaoPeng2003}, we know that $K=4/\alpha$ for some positive integer $\alpha$. Then, by Minding's theorem, we can choose a complex coordinate $z$ on $S^2\setminus\{\infty \}$ such that 
	 \begin{equation}\label{constantcurvature}
	 	X^*h_{rn}=\frac{\alpha}{(1+z\bar{z})^2}dzd\bar z= \partial_z\partial_{\bar z}  \log (1+z\bar{z})^{\alpha}dzd\bar z.
	 \end{equation}
For each $j\in\{0,\ldots,r-1\}$, we compose the holomorphic map $X_j$ with the Pl\"{u}cker embedding  
$\iota: G_{k_j}(\C^n)\to \CP^{n_j-1}$, with $n_j={n \choose k_j}$,  to obtain  a holomorphic map 
$\sigma_j:S^2\to\CP^{n_j-1}$. This holomorphic map $\sigma_j$ satisfies
${\sigma_j}^*h_{n_j}=X_j^*h_n$ (see \cite[\S 3]{JiaoPeng2003}). 
Let $\widehat \sigma_j:\C\to \C^{n_j}$ be a polynomial holomorphic local section of $\sigma_j$, with $\C\cong S^2\setminus\{\infty\}$. Without loss of generality, we can assume that $\widehat \sigma_j$ is nowhere vanishing, as we can always remove the greatest common divisor of its components.  Set $\beta_j=\|\widehat \sigma_j\|^2,$ which is (nonconstant) polynomial in $z$ and $\bar z$. From  \cite[\S 3]{JiaoPeng2003}, we have
\begin{equation}\label{eq:logbetaj}
	 X_j^*h_n=\partial_z \partial_{\bar z} \log \beta_j dzd\bar z.
\end{equation}    
Hence, in view of \eqref{flagmetric} and \eqref{eq:logbetaj}, 
\begin{equation}\label{metric3}
	X^*h_{rn}=\partial_z \partial_{\bar z} \log\left(\beta_0\ldots  \beta_{r-1}\right) dzd\bar z.
\end{equation}
Comparing \eqref{constantcurvature} and \eqref{metric3}, we deduce that the following holds on $\C$:
\begin{equation*}
	\partial_z\partial_{\bar z}  \log \frac{ \beta_{0}\ldots \beta_{r-1}}{(1+z \bar{z})^{\alpha}}=0.
\end{equation*}
Hence, $\log \frac{ \beta_{0}\ldots \beta_{r-1}}{(1+z \bar{z})^{\alpha}}$ is a harmonic function on $\mathbb{C},$ which implies that there exists an entire function $f$ such that 
$\beta_{0}\ldots \beta_{r-1}= (1+z\bar z)^{\alpha} e^{f}\overline{e^{f}}.$ Since the left-hand side is a polynomial in \(z\) and \(\bar z\), the entire function \(e^f\) must be constant. Since  $1+z\bar z$ is irreducible in $\C[z,\bar z]$ and each $\beta_j$ is polynomial in $z$ and $\bar z$, we must have 
	$\beta_j(z,\bar z)=C_j(1+z\bar z)^{\alpha_j}$
for some positive real constant $C_j$ and positive integer $\alpha_j$, with $\alpha=\alpha_0+\ldots+\alpha_{r-1}$.
It follows from \eqref{eq:logbetaj}  that
$$X_j^*h_{n}= \frac{\alpha_j}{(1+z \bar{z})^2}dz d\bar z, \quad \text{for each}\, j\in\{0,\ldots,r-1\},$$
which shows that each $X_j$ is an immersion with constant curvature.

Conversely, if each $X_j^*h_n$ is a positive multiple of the round metric on $S^2$, then \eqref{flagmetric} immediately implies that the same is true for $X^*h_{rn}$. Hence, by Lemma \ref{metricslemma}, $\psi$ is a minimal immersion with constant curvature.
\end{proof}

\section{Constantly curved minimal immersions of  $S^2$ into $\U(3)$}\label{immersioninu3}
In this section, we investigate minimal immersions of constant curvature from $S^2$ into $\U(3)$. In this setting, the uniton number of a minimal immersion $\psi : S^2 \to \U(3)$  can only be $1$ or $2$. We will denote by 	$V_0^n,\ldots, V^n_n :S^2 \rightarrow \mathbb{C}P^n$ the \emph{Veronese sequence}, defined as in \cite[Theorem 5.2]{Bolton}.

\subsection{Uniton number $r=1$}

\begin{theorem}\label{uniton1}
	Let $\psi:S^2 \to \U(3)$ be a  minimal immersion with constant curvature and uniton number $1$. Then $\psi$ is unitarily congruent to one of the following minimal immersions:
	\begin{enumerate}
		\item $\pi_{V_0^2}-\pi_{V_0^2}^\perp$;
		\item $\pi_{V_0^2\oplus V_1^2}-\pi_{V_0^2\oplus V_1^2}^\perp$;
		\item $\pi_{V_0^1\oplus \C}-\pi_{V_0^1\oplus \C}^\perp$, where  $\C\cong(V^1_0\oplus V^1_1)^\perp\subset \C^3$. 
	\end{enumerate}
\end{theorem}
\begin{proof}
 When the uniton number is $1$, any normalized extended solution is $S^1$-invariant; more precisely, it is of the form $W= X_0 \oplus\lambda \mathcal{H}_+,$  where $X_0:S^2\to G_k(\C^3)$ is a linearly full holomorphic map for some $k \in \{1,2\}$. By Theorem \ref{thm:s1invariant}, $X_0$ is a linearly full holomorphic immersion with constant curvature.  

For $k=1$, by Calabi's  rigidity result \cite[Theorem 1.1]{Lawson}, $X_0$ is unitarily congruent to $V_0^2$. For $k=2$, $X_0^\perp:S^2\to\CP^2$ is an antiholomorphic immersion with constant curvature, not necessarily linearly full, hence, after unitary transformation, either $X_0^\perp=V_2^2$, in which case $X_0=V_0^2\oplus V_1^2$, or $X_0^\perp=V_1^1$, in which case $X_0=V_0^1\oplus \C$, where  $\C\cong(V^1_0\oplus V^1_1)^\perp$. 
\end{proof}

\subsection{Uniton number $r=2$}
We first describe the structure of a normalized extended solution  of uniton number $2$. 
\begin{lemma}\label{unitonnumber2}\begin{enumerate}
		\item Let $X_0\subset X_1$ be two linearly full holomorphic subbundles of  $S^2 \times \C^3$,  
		such that: $\mathrm{rank}\,X_0=1$; $\mathrm{rank}\,X_1=2$; and
		$\partial_z\Gamma(X_0)\subset 	\Gamma(X_1)$. Let $R$ be a holomorphic line subbundle of   $X_0\oplus \lambda X_1^\perp$, that is, 
		$$\pi_{R^\perp \cap (X_0\oplus \lambda X_1^\perp)}\partial_{\bar z}\Gamma(R)=0,$$ such that $\pi_{X_0}(R)=X_0$ almost everywhere.
		Then $W=X\oplus\lambda^2\mathcal{H}_+$, with $X=R\oplus \lambda X_1$,
		is a normalized extended solution on $S^2$ of uniton number $2$. 
		\item Conversely, any normalized extended solution $W= X\oplus \lambda^2\mathcal{H}_+$ on $S^2$
		of uniton number $2$ is of this form.
\end{enumerate}\end{lemma}

\begin{proof} Let $W$ be as in the first statement. Then $W$  satisfies the shift-invariance property \eqref{shiftinvariance} because $X_0\subset X_1$.  The holomorphicity of $W$ follows immediately from the holomorphicity of $X_0$, $X_1$,  and $R$. The assumption $\partial_z\Gamma(X_0)\subset\Gamma(X_1)$ implies that $W$ satisfies $\lambda \partial_z\Gamma(W)\subset \Gamma(W)$. Finally, since $X_0$ and $X_1$ are linearly full,       
$\bigcup_{z\in S^2}W(z)$ spans $\mathcal{H}_+$. Hence $W$ is a normalized extended solution of uniton number $r=2$. 

Conversely, assume that $W= X\oplus \lambda^2{\mathcal{H}}_+$ is a normalized extended solution on $S^2$ of uniton number $2$. Consider the  projections $p_0,p_1: {\C}^3\oplus \lambda  {\C}^3\to {\C}^3$ defined by $p_0(a+\lambda b)=a$ and $p_1(a+\lambda b)=b$. Since $X\subset  {\C}^3\oplus \lambda  {\C}^3$ is holomorphic, the vector subbundles  $X_0=p_0(X)$ and $X_1=p_1(X\cap \lambda {\C}^3)$  of $S^2\times{\C}^3$  both have constant rank away from a finite subset of $S^2$. We can fill out the zeros to obtain globally defined holomorphic vector subbundles of $S^2\times \C^3$, which we continue to denote by $X_0$ and  $X_1$. It follows immediately that
$X_0\subset X_1$, since $\lambda X\subset X \mod\lambda^2 \mathcal{H}_+$. Moreover,	$\partial_z\Gamma(X_0)\subset 	\Gamma(X_1)$, since $\lambda \partial_z\Gamma(X)\subset \Gamma (X)\mod \lambda^2\mathcal{H}_+$. 

Let $R$  be the  orthogonal complement of $\lambda X_1$ in $X$. Thus $R=X\cap (X_0\oplus \lambda X_1^\perp)$ and
 $X=R\oplus \lambda X_1$.   The holomorphicity of $W$ implies that $R$ is  holomorphic in $X_0\oplus \lambda X_1^\perp$.
 By construction,  $\pi_{X_0}(R)=p_0(X)=X_0$ almost everywhere. Finally, since $W$ is normalized, both $X_0\subset X_1$ are linearly full,  which implies that $X_0\subsetneq X_1$. Indeed, we have $\partial_{\bar z}\Gamma(X_0)\subset 	\Gamma(X_0)$, since  $X_0$ is holomorphic. If additionally $X_0=X_1$, then 	$\partial_z\Gamma(X_0)\subset 	\Gamma(X_1)=\Gamma(X_0)$, which means that $X_0$ is parallel, hence constant, contradicting linear fullness. Moreover, \(\operatorname{rank}X_1<3\), since otherwise \(X_1=\mathbb C^3\) and \(W\) would have uniton number at most \(1\). Therefore, $\mathrm{rank}\,X_0=1$ and $\mathrm{rank}\,X_1=2$.
		\end{proof}

Let $\psi: S^2\to \U(3)$ be a harmonic map of uniton number 2, with normalized extended solution  $W=X\oplus\lambda^2\mathcal{H}_+$, where $X=R\oplus\lambda X_1:S^2\to G_3(\C^3\oplus \lambda \C^3)$ is defined  as in Lemma \ref{unitonnumber2}. We consider the following orthogonal decompositions of $X$ and $X^\perp$ into line subbundles: 
\begin{equation}\label{mapX}
X=R\oplus \lambda X_0\oplus \lambda(X_1\cap X_0^\perp), \quad X^\perp=(X_1\cap X_0^\perp)\oplus \left((X_0\oplus\lambda X_1^{\perp})\cap {R}^\perp\right)\oplus X_1^\perp.\end{equation}
Observe that   the  $\partial_z$-Gauss bundles of $X_0$ are given by $$G^{(1)}(X_0)=X_1\cap X_0^\perp,\quad  G^{(2)}(X_0)=X_1^\perp,\quad G^{(3)}(X_0)=\{0\}.$$
We now determine the $\partial_z$-Gauss bundles of $X$. It follows from the construction in Lemma \ref{unitonnumber2} that
\begin{align}
	\label{gamma1} & \partial_z\Gamma(\lambda X_0)\subset \Gamma(X);\\
\label{gamma2}&	 \partial_z\Gamma(\lambda (X_1\cap X_0^\perp)) \subset \Gamma(X)\oplus \Gamma\left((X_0\oplus\lambda X_1^{\perp})\cap {R}^\perp\right); \\ \label{gamma3}& \partial_z\Gamma(R)\subset \partial_z\Gamma(X_0\oplus \lambda X_1^\perp)\subset \Gamma(X)\oplus \Gamma(X_1\cap X_0^\perp)\oplus\Gamma \left((X_0\oplus\lambda X_1^{\perp})\cap {R}^\perp\right).
\end{align}
We have therefore the following diagram:
\begin{equation}
	\begin{gathered}\label{diag:ref}
		\xymatrixrowsep{1.5pc}{\xymatrix{
				{R} \ar[rd]  \ar[r] &X_1\cap X_0^\perp \ar[r]  & X_1^\perp  \\	\lambda(X_1\cap X_0^\perp)  \ar[u]  \ar[r] &(X_0\oplus\lambda X_1^{\perp})\cap {R}^\perp  \ar[u] & \\
				\lambda X_0 \ar[u]& &
		}}
	\end{gathered}.
\end{equation}
Recall from \cite{burstall-wood,PW1} that, in a  diagram like this, each arrow between two orthogonal bundles $\varphi_1,\varphi_2$ corresponds to the vector bundle homomorphism $A'_{\varphi_1,\varphi_2}\in \Hom(\varphi_1,\varphi_2)$ defined locally by $ A'_{\varphi_1,\varphi_2}(s)=\pi_{\varphi_2}\partial_zs$, for each $s\in \Gamma(\varphi_1)$.  No arrow is shown whenever the corresponding homomorphism is known to be zero.  In \eqref{diag:ref}, observe that all the horizontal arrows are nonzero as bundle homomorphisms. Indeed, the arrow from $X_1\cap X_0^\perp$ to $X_1^\perp$ is nonzero because $X_0$ is linearly full. 
The arrow from $R$ to $X_1\cap X_0^\perp$ 
is also nonzero, since
\begin{equation}\label{arx}
A'_{R,X_1\cap X_0^\perp}(s)=A'_{X_0}(s_0)
\end{equation}
for every section $s=s_0+\lambda s_1$ of $R$, and
$A'_{X_0}$ is nonzero. Finally, observe that, since the image of $A'_{\lambda(X_1\cap X_0^\perp)}$  is $$\lambda X_1^\perp\subset X_0\oplus \lambda X_1^\perp=R\oplus \left((X_0\oplus\lambda X_1^{\perp})\cap {R}^\perp\right),$$ we have
 $$A'_{\lambda(X_1\cap X_0^\perp)}= A'_{\lambda(X_1\cap X_0^\perp),R}+A'_{\lambda(X_1\cap X_0^\perp),(X_0\oplus\lambda X_1^{\perp})\cap {R}^\perp}.$$
Since $\pi_{X_0}(R)=X_0$ almost everywhere, $\lambda X_1^\perp$ is not contained in $R$, hence  the second homomorphism on the right-hand side,  corresponding to
 the arrow from $\lambda(X_1\cap X_0^\perp)$  to $(X_0\oplus\lambda X_1^{\perp})\cap {R}^\perp$, is also nonzero.  We conclude that $$G^{(1)}(X)=\left(X_1\cap X_0^\perp\right) \oplus\big((X_0\oplus\lambda X_1^{\perp})\cap{R}^{\perp}\big),\,\, G^{(2)}(X)=X_1^\perp,\,\, G^{(3)}(X)=\{0\}.$$

\begin{remark}\label{remark}
	In the setting of Lemma \ref{unitonnumber2}, let $f_0$ be a local  polynomial lift of $X_0$ on $\C\cong S^2\setminus\{\infty\}$.  Since $X_0$  is linearly full, we have $X_0=\Span \{f_0\}$, $X_1=\Span \{f_0,\partial_z f_0\}$ and $\C^3=\Span \{ f_0, \partial_z f_0, \partial^2_z f_0 \}$ away from a finite subset of ${S}^2$.  In particular,
	$\pi_{X_1^\perp}(\partial_z^2f_0)$
	is a local generator of $X_1^\perp$. Since $\pi_{X_0}(R)=X_0$ almost everywhere and $R$ is a holomorphic line subbundle of $X_0\oplus \lambda X_1^\perp$, 
	locally away from a finite
	subset we may choose a generator of $R$ of the form $$
	f_0+Q\lambda\pi_{X_1^\perp}(\partial_z^2f_0),$$
	for some meromorphic function $Q$.
Then
	\begin{equation}\label{Q}
		X= \Span\{\lambda f_0,\lambda \partial_z f_0,f_0+Q\lambda \partial^2_zf_0\}
	\end{equation}
	away from a finite subset of ${S^2}$. When $Q=0$, we have $X=X_0\oplus \lambda X_1$, that is, the extended solution is $S^1$-invariant.      
\end{remark}

\begin{lemma}\label{lem:tounram} Let $W= X\oplus \lambda^2\mathcal{H}_+$ be a normalized extended solution on $S^2$
	of uniton number $2$, and let $X_0 \subset X_1$ be as in  Lemma \ref{unitonnumber2}. If $X$  is totally unramified, then: 
	\begin{enumerate}\item 
$X_0$ is totally unramified (hence it has degree $2$);\item $\pi_{X_0}(R)=X_0$ everywhere on $S^2$; \item 
if $f_0$ is a quadratic polynomial lift of $X_0$ on $\C\cong S^2\setminus\{\infty\}$, then the meromorphic function $Q$ in \eqref{Q} is polynomial, with $\deg Q\leq 2.$   
  \end{enumerate}
	\end{lemma}
\begin{proof}
	In diagram \eqref{diag:ref}, the horizontal arrow  from $R$ to $X_1\cap X_0^\perp$  corresponds to the line bundle homomorphism $A'_{R,X_1\cap X_0^\perp}\in\Hom\left(R, X_1\cap X_0^\perp \right)$, which satisfies \eqref{arx}.
	Since $G^{(1)}(X_0)=X_1\cap X_0^\perp$ and $G^{(2)}(X_0)=X_1^\perp$, the horizontal arrow from $X_1\cap X_0^\perp$ to $X_1^\perp$  is given by  $A'_{G^{(1)}(X_0)}\in \Hom (G^{(1)}(X_0), G^{(2)}(X_0) )$. 
	If $X$ is totally unramified, these horizontal arrows are unramified, that is, the bundle homomorphisms
	$A'_{R,X_1\cap X_0^\perp}$ and $A'_{G^{(1)}(X_0)}$ are nowhere vanishing on $S^2$. Hence, in view of \eqref{arx}, $A'_{X_0}$ and $A'_{G^{(1)}(X_0)}$ are nowhere vanishing on $S^2$, that is, $X_0$ is totally unramified, which implies that $X_0$ has degree $2$ (see  \cite[Section 3]{Bolton}). 
	
	If the local section $s=s_0+\lambda s_1$ of $R$ in \eqref{arx} is  nowhere vanishing, then $s_0$ is  also nowhere vanishing. Indeed, 
if $s_0$ vanishes at some point, then  $A'_{R,X_1\cap X_0^\perp}(s)$ would vanish at that same point, contradicting the assumption that $X$ is totally unramified. Hence,   $\pi_{X_0}(R)=X_0$ everywhere on $S^2$.

	 Let $f_0$ be a  quadratic polynomial lift of $X_0$ on $\C\cong S^2\setminus\{\infty\}$.   Since $X_0$  is linearly full and totally unramified, we have $X_0=\Span \{f_0\}$, $X_1=\Span \{f_0,\partial_z f_0\}$ and $\C^3=\Span \{f_0,\partial_z f_0, \partial^2_z f_0 \}$ everywhere on $\C$.  Here, $\partial_z^2f_0$ is a nonzero constant. Since   $\pi_{X_0}(R)=X_0$ on $\C\cong S^2\setminus \{\infty\}$, the meromorphic function $Q$ in \eqref{Q} has no finite poles, hence it is polynomial. 
	To determine its degree, consider $w=1/z$. Since $f_0$ is quadratic,
	  $
	 \widetilde f_0(w)=w^2f_0(1/w)
	 $
	 is a holomorphic lift of $X_0$ near $w=0$. Thus, by \eqref{Q}, the corresponding local generator of $R$, modulo $\lambda X_1$, is
	  $\widetilde f_0(w)+w^2Q(1/w)\lambda\,\partial_z^2f_0. $
	  Since $\pi_{X_0}(R)=X_0$ also at $\infty$, $w^2Q(1/w)$ is holomorphic at $w=0$. Hence
	  $ \deg Q\leq2. $
	  \end{proof}

\begin{theorem}\label{thm:unramified2}
Let $\psi:S^2 \to \U(3)$ be a  minimal immersion with constant curvature, uniton number $2$, and normalized extended solution $W= X\oplus \lambda^2\mathcal{H}_+$. Let $X_0\subset X_1$ and $R$ be  defined as in  Lemma \ref{unitonnumber2}. If $X_0$ is totally unramified, then $\psi$ is $S^1$-invariant. Consequently, $\psi$ is unitarily congruent to $\pi_{V_1^2}-\pi_{V_1^2}^\perp$. 
\end{theorem}

\begin{proof}
		By Lemma \ref{unitonnumber2},
	$X=R\oplus \lambda X_1$,
	$R\subset X_0\oplus \lambda X_1^\perp,$
	and $\pi_{X_0}(R)=X_0$ almost everywhere.
	Suppose that, at some point $p\in S^2$,
	$\pi_{X_0}\left(R(p)\right)=0.$
	Since $R(p)$ is a complex line contained in
	$X_0(p)\oplus \lambda X_1^\perp(p),$
	it follows that
	$R(p)=\lambda X_1^\perp(p).$
	Hence
	$X(p)
	=
	R(p)\oplus \lambda X_1(p)
	=
	\lambda X_1^\perp(p)\oplus \lambda X_1(p)
	=
	\lambda\C^3.$
	Thus every point at which $\pi_{X_0}(R)$ vanishes is mapped by $X$ to the
	same point $\lambda\C^3$ of the Grassmannian.

	Let $\sigma:S^2\to\mathbb{CP}^{19}$ be the composition of $X:S^2\to G_3(\mathbb C^3\oplus\lambda\mathbb C^3)$ with the Plücker embedding. By Lemma 3.1, $X$ is a holomorphic immersion with constant curvature, and the Plücker embedding preserves the induced metric. Hence $\sigma$ is also a holomorphic immersion with constant curvature. By Calabi's rigidity theorem,  $\sigma$ is unitarily congruent to a Veronese curve and is therefore injective. Since every zero of $\pi_{X_0}(R)$ is mapped by $X$ to $\lambda\mathbb C^3$, there can be at most one such point. Thus $\pi_{X_0}(R)=X_0$ everywhere on $S^2$, except possibly at one point. If such a point exists, choose it to be $\infty$; otherwise choose any point as $\infty$. Choose the complex coordinate $z$ on $S^2\setminus\{\infty\}$ so that the constant-curvature metric induced by $X$ is a positive multiple of $(1+z\bar z)^{-2}dz\,d\bar z$.

	 Let $f_0$ be a  quadratic  polynomial  lift of $X_0$ on $\C\cong S^2\setminus \{\infty\}$; such a lift exists since \(X_0\) is totally unramified and hence has degree $2$. Then \eqref{Q} holds for some meromorphic function $Q$. Arguing as in the last paragraph of the proof of Lemma \ref{lem:tounram}, we see that $Q$ is a polynomial, since   $\pi_{X_0}(R)=X_0$  everywhere on $\C$.  We shall prove that, under the assumptions of the theorem, $Q=0$ (that is, $\psi$ is $S^1$-invariant); then, by Theorem \ref{thm:s1invariant} and Calabi's rigidity theorem, $X_0=\Span\{f_0\}$ is unitarily congruent to the Veronese map $V_0^2:S^2\to\CP^2$.

The map
$\sigma$ admits a   local holomorphic  polynomial lift $\widehat \sigma:\C\to \bigwedge^3(\C^3 \oplus \lambda \C^3)$ of the  form
\begin{equation}
	\widehat\sigma =  (\lambda f_0)\wedge (\lambda \partial_z f_0) \wedge f_0    +\widehat Qc_0^{-1}\left((\lambda f_0)\wedge (\lambda \partial_z f_0)\wedge  (\lambda \partial^2_z f_0)\right),
\end{equation}
where $c_0=\|  f_0\wedge \partial_z f_0 \wedge \partial^2_zf_0 \|$ is a positive constant and $\widehat Q$ is  the polynomial  defined by $\widehat Q={c_0} Q$.  Hence
	$$\|\widehat\sigma\|^2=\| f_0\wedge \partial_z f_0\|^2\|f_0\|^2+|\widehat Q|^2.$$
	Since $\psi$ has constant curvature,  Lemma \ref{metricslemma} implies that the holomorphic map $X$ is an immersion of constant curvature, and it follows that $\|\widehat\sigma\|^2=d(1+z\bar z)^m$, where $d$ is a positive real number and $m$ is a  positive integer. 
Set
 \begin{equation}\label{Pz0}
 P(z,\bar z)=\sum P_{kl}z^k\bar z^l:= 	|\widehat Q|^2 =d(1+z\bar z)^m -  \| f_0\wedge \partial_z f_0\|^2\|f_0\|^2.
 \end{equation}

Since $X_0$ is linearly full, after a unitary congruence and a constant rescaling of the polynomial lift, we may assume that $f_0$ takes the form
\begin{equation}\label{f_0}
	f_0(z)=(1+a_1z +a_2z^2, b_1 z+b_2 z^2,c_2z^2),
\end{equation}
 up to unitary congruence, with  $b_1 c_2\neq 0$. By applying a unitary change of variable of the form $z\mapsto e^{i\alpha}z$, the coefficient  $a_1$ can be made real. Finally,   
an appropriate 
diagonal unitary transformation $\mathrm{diag}(1,e^{i\beta} ,e^{i\gamma})\in \U(3)$ makes $b_1$ and $c_2$ positive real numbers. Thus, without loss of generality, we  assume the following:
\begin{equation}\label{ansatz}
	a_1\in\mathbb{R};\,\,  a_2,b_2\in\C; \,\,   b_1,c_2,d>0.
\end{equation}

We claim that $m=4$.  Considering  the ansatzes \eqref{f_0} and \eqref{ansatz}, $\| f_0\wedge \partial_z f_0\|^2\|f_0\|^2$ is a polynomial $\sum_{k,l=0}^4S_{kl}z^k\bar z^l$,  with leading coefficient $S_{44}>0.$   If $m \leq  3$ in \eqref{Pz0}, then the leading coefficient  of $P(z,\bar z)$ would be $P_{44}$; moreover, we would have $P_{44}<0$,  which is impossible since  $P(z,\bar z)=|\widehat Q(z)|^2$. Hence $m\geq 4$.  Since $P(z,\bar z)$ is separable, we  have 
\begin{equation}\label{Pz}
	P(z,\bar z)P(0,0)=P(z,0)P(0, \bar z).
	\end{equation}
Observe that  $P(z,0)$ and $P(0,\bar z)$ are polynomials of degree at most $4$ in $z$ and $\bar z$, respectively. Hence, if $P(0,0)\neq 0$,  \eqref{Pz} implies that  $m\leq 4$, and this implies $m=4$.

If $P(0,0)=0$, then $\widehat Q(0)=0$, since $P(z,\bar z)=|\widehat Q(z)|^2$. Hence $P(z,0)=0$. Setting $\bar z=0$ in \eqref{Pz0}, we have
\begin{align*}
	P(z,0)=d-b_1^2 &- (a_1 b_1^2  + 2 b_1 b_2) z - 3 a_1 b_1 b_2 z^2 \\&- (2 a_2 b_1 b_2 - a_1 a_2 b_1^2  + 
	a_1^2 b_1 b_2 )z^3   - (a_1 a_2 b_1 b_2 - a_2^2 b_1^2 )z^4.
\end{align*}
Since $b_1\neq0$, the identity $P(z,0)=0$ gives
$a_1=a_2=b_2=0$ and $d=b_1^2$.
Substituting these values into \eqref{Pz0} and  \eqref{f_0}, we obtain
\begin{align*}P(t):=P(z,\bar z)=d(1+t)^m-\left(d + (d^2 + 4 c_2^2)  t+ 6 d c_2^2  t^2+ c_2^2(d^2 + 4 c_2^2) t^3+d c_2^4 t^4\right),\end{align*} with $t=z\bar z$.  
Since $P(z,\bar z)=|\widehat Q(z)|^2$ is radial, $\widehat Q(z)$ must be a monomial, say
$\widehat Q(z)=qz^r.$ Thus $P(t)=|q|^2t^r$. If $m\geq 6$, then, since the coefficient of $t^m$ in $P(t)$ is $d>0$, we must have $r=m$. Hence the coefficient of $t^5$ in $P$ must vanish. However, this coefficient is
$d\binom{m}{5}>0$, a contradiction. Therefore $m\leq 5$.
It remains to exclude $m=5$. If $m=5$, then $r=5$, and the coefficients of $t^2$ and $t^4$ in $P(t)$ must vanish. It is easy to check that this is impossible. Hence $m\neq 5$, and therefore $m=4$.

 Consider   
\begin{equation}\label{P(z)}
P(z,\bar z)=d(1+z\bar z)^4 -  \| f_0\wedge \partial_z f_0\|^2\|f_0\|^2.
\end{equation}
Next, we show that,  under the ansatzes \eqref{f_0} and \eqref{ansatz}, the separability of $P(z,\bar z)$ implies  $a_1=a_2=b_2=0$, $c_2=1$, $d=2$  and $b_1=\sqrt 2$, so that $X_0=\Span\{f_0\}$ is the Veronese map $V_0^2$. 
For this purpose, we use the following necessary condition for separability:
\begin{equation}\label{polynomialR}
R(z,\bar z):=\sum_{k,l} R_{kl} z^k\bar z^l=P(z,\bar z)\partial_z\partial_{\bar z}P(z,\bar z)-\partial_zP(z,\bar z) \partial_{\bar z}P(z,\bar z)=0.
\end{equation}
The coefficients $R_{kl}$ can be found by long  but straightforward computations. These and other straightforward  computations in this proof were carried out using \emph{Wolfram Mathematica}.   The \emph{Wolfram Mathematica} notebook containing the computations is available at  the GitHub repository \cite{Git}.
We seek all values of  $a_1,a_2,b_1,b_2,c_2,d$ satisfying \eqref{ansatz} for which  the system $R_{kl}=0$ (for all $k,l$)  holds. 
\\

\textbf{Case  $a_1a_2b_2\neq 0$.} For general values of the parameters, we have: 
\begin{align}\label{R50}
	R_{50} = 6 b_1^2 \big(b_1b_2(a_2b_1-a_1b_2)^2 +2a_1a_2^2c_2^2\big);\,\, R_{60}	=  b_1^2 \big(-b_1(a_2b_1-a_1b_2)^3+4a_2^3c_2^2 \big).
\end{align}
 Since, by \eqref{ansatz}, $b_1c_2\neq 0$, from $R_{50}=0,R_{60}=0$ we deduce that either $a_1a_2b_2=0$ (we consider this case later) or  
\begin{equation}\label{b2}b_2=\frac{a_1a_2b_1}{a_1^2-2a_2}
	\end{equation}
and 
\begin{equation}\label{b14}b_1^4=\frac{(2a_2-a_1^2)^3}{2a_2^3}c_2^2. 
\end{equation}
Since $a_1$, $b_1$ and $c_2$ are real, we get  from \eqref{b14} that
\begin{equation}\label{a2a2}
\frac{\overline{a}_2}{2\overline a_2-a_1^2}= \frac{\omega {a}_2}{2 a_2-a_1^2},
\end{equation}
with $\omega$ a cube root of unity. Hence, from \eqref{b2}, we obtain
\begin{equation}\label{b2ob2}
	\overline{b}_2=\omega b_2.
\end{equation}
On the other hand, from \eqref{a2a2}, we get
\begin{equation}\label{a2}
\overline{a}_2=\frac{a_1^2 a_2 \omega}{a_1^2 +2 a_2 ( \omega-1)}.
\end{equation}

Using \eqref{b2} to eliminate $b_2$, long but straightforward  computations yield
\begin{align*}R_{40}=12 a_2 b_1^2\frac{ (3 a_1^2 - 2 a_2) a_2^3 b_1^4 + (a_1^2 - 2 a_2)^3 (a_1^2 + a_2) c_2^2 +2 d (a_1^2 - 2 a_2)^2 a_2^2}{(a_1^2 - 2 a_2)^3}.\end{align*}
Hence, setting $R_{40}=0$ and using \eqref{b14},
\begin{equation}\label{d}
d= \frac{c_2^2(a_1^2-2a_2)(a_1^2-4a_2)}{4 a_2^2}.
\end{equation}

Using  \eqref{b2}, \eqref{b2ob2}, \eqref{a2} and \eqref{d}, the coefficient $R_{61}$ simplifies to
$$R_{61}=\frac{12 a_1 a_2^3 b_1^2 \big((a_1^2 - 2 a_2)^2 ((a_1^2 - 2 a_2)^2 + a_1^2 b_1^2) c_2^2 + 
4 a_2^4 b_1^4 \omega\big)}{(a_1^2 - 2 a_2)^4}.$$
Recall that $b_1c_2\neq 0$ and we are considering the case $a_1a_2b_2\neq 0$. Solving the equation $R_{61}=0$ for $b_1^2$ and taking \eqref{b14} into account, we obtain:
\begin{equation}\label{b12}
	b_1^2=-\frac{(a_1^2 - 2 a_2) (a_1^2 - 2 a_2 ( 1+\omega ))}{a_1^2}.
\end{equation}
We can now use \eqref{b14}, \eqref{d}  and \eqref{b12} in order to express $c_2^2$ and $d$ in terms of $a_1$, $a_2$ and $\omega$:
\begin{equation}\label{c2d}
c_2^2=\frac{2 a_2^3 (a_1^2 - 2 a_2 (1 + \omega))^2}{a_1^4( 2  a_2-a_1^2 )},\quad d=\frac{a_2(  4 a_2-a_1^2)  (a_1^2 - 2 a_2 (1 + \omega))^2}{2 a_1^4}.
\end{equation}

 The coefficient $R_{10}$ then simplifies to
\begin{align*}R_{10}= -\frac{6 (a_1^2 - 4 a_2) a_2 (a_1^2 - 2 a_2 (1 + \omega))^3 (a_1^2 + 
	2 a_2^2 (1 + \omega))}{a_1^5}.
\end{align*}          
 If   $a_1^2 - 2 a_2 (1 + \omega)=0$ or $a_1^2 - 4 a_2=0$, then, by \eqref{b12} and \eqref{c2d}, $b_1=0$ or $d=0$, which contradicts \eqref{ansatz}. 
  Hence, for $R_{10}=0$, it remains to consider $a_1^2 + 2 a_2^2 (1 + \omega)=0$. Under this assumption, the identities \eqref{b12} and \eqref{c2d} simplify to
  \begin{equation}\label{c2dsimp}
  b_1^2=2 (1 + a_2) (1 + a_2 (1+\omega) ),\,\, 	c_2^2=\frac{(1 + a_2)^2}{1 + a_2 (1+\omega)} ,\,\,d=(1 + a_2)^2 (2 + a_2 (1+\omega)).
  \end{equation}
  Moreover, $R_{51}$ can be expressed solely in terms of $a_2$ and $\omega$:
  $$R_{51}=-48 a_2^3 (1 + a_2)^3 (-1 + \omega).$$
 Then, the equation $R_{51}=0$ implies that  $a_2=-1$ or $\omega=1$. If $a_2=-1$, then \eqref{c2dsimp} yields $c_2=0$, which  contradicts \eqref{ansatz}. Hence, we must have $\omega=1$.  Since $a_1$ is real, the condition $a_1^2+2a_2^2(1+\omega)=0$ with $\omega=1$ forces $a_2$ to be  purely imaginary.  On the other hand, by \eqref{c2dsimp},  $b_1^2=2(1+a_2)(1+2a_2)$. Since $b_1$ is real, the right-hand side must also be real. Since $a_2$ is purely imaginary, this forces $a_2=0$, contradicting the assumption \(a_1a_2b_2\neq0\).  Therefore, the case $a_1a_2b_2\neq0$ yields no solutions of the system $R_{kl}=0$. We now turn to the case $a_1a_2b_2=0$.  
  \\

\textbf{Case $a_1a_2b_2=0$.} Under the conditions \eqref{ansatz}, we now assume that $a_1a_2b_2=0$. We first prove that, if any of the coefficients $a1,a_2,b_2$ vanishes, the other two vanish as well: 

\begin{enumerate}
\item[1)]	If $a_1=0$, then, in view of \eqref{R50}, $R_{50}=0$ yields $a_2b_2=0$, and we have:

1a)	If $a_1=a_2=0$, then $R_{30}=8 b_1^3 b_2^3$, hence $R_{30}=0$ yields $b_2=0$.

1b) Assume that $a_1=b_2=0$. By \eqref{R50}, 
the equation	$R_{60}=0$ implies that  $a_2=0$ or $c_2= \frac{b_1^2}{2}$.
	If  $c_2=\frac{b_1^2}{2}$, then
	$R_{40}=6 a_2^2 b_1^2 (b_1^4 - 2 d)$, hence
	 $R_{40}=0$ gives $a_2=0$ or  $d=\frac{b_1^4 }{2}$. If $c_2= \frac{b_1^2}{2}$ and $d=\frac{b_1^4 }{2}$, then
$R_{51}=6 a_2^2 b_1^6 (-2 + b_1^2)$, hence $R_{51}=0$ yields $a_2=0$ or $b_1=\sqrt{2}$. For $b_1= \sqrt{2}$, $c_2= \frac{b_1^2}{2}=1$ and $d=\frac{b_1^4 }{2}=2$, we have $R_{62}=48 a_2^3\overline{a}_2$, hence $R_{62}=0$ forces  $a_2=0$. 

\item[2)] 	If $a_2=0$, then,  in view of \eqref{R50}, $R_{50}=0$ gives $a_1b_2=0$. If $a_1=0$, then we fall in the previous case 1a), and $b_2=0$ as well. So we assume $a_2=b_2=0$.   In this case, $R_{30}=4 a_1^3 b_1^2 c_2$, hence $R_{30}=0$ implies $a_1=0$.
\item[3)] 	If $b_2=0$, then,  in view of \eqref{R50}, $R_{50}=0$ gives $a_1a_2=0$, and we fall in one of the previous cases. 
	\end{enumerate}
	Hence the system $R_{kl}=0$ implies that $a_1=a_2=b_2=0$.  With these values, $$R_{11}=24 (b_1^2 - d) (b_1^2 c_2^2 - d),$$ 
	hence $R_{11}=0$ gives $d=b_1^2$ or $d=b_1^2c_2^2$, and we have:
	
	 If $d=b_1^2$ (together with $a_1=a_2=b_2=0$), then  $$R_{55}=24 b_1^4 (-1 + c_2^2)^2 (1 + c_2^2).$$
	Since $c_2$ is real, $R_{55}=0$ gives $c_2=1$. We now have  
	$R_{33}=4 (-2 + b_1^2)^4$, and consequently $R_{33}=0$ forces $b_1=\sqrt{2}$.

	 If $d=b_1^2c_2^2$  (together with $a_1=a_2=b_2=0$), then 
	$$R_{66}=b_1^2 c_2^4 ( c_2^2-1) (-4 b_1^2 + b_1^4 + 4 c_2^2),\,\, R_{00} =-b_1^2 ( c_2^2-1) (b_1^4 - 4 (-1 + b_1^2) c_2^2).$$
If $c_2=1$, then $d=b_1^2$, and the preceding case gives
$b_1=\sqrt2$. If $c_2\neq1$, cancelling the common factor
$c_2^2-1$ in $R_{66}=R_{00}=0$ gives $b_1^2=2$ and
$c_2^2=1$, a contradiction.
	
	We conclude that, under the assumptions of the theorem, $a_1=a_2=b_2=0$, $c_2= 1$, $d=2$, and $b_1=\sqrt 2$ . This implies that $f_0$ in \eqref{f_0} takes the form   
	 $f_0(z)=(1, \sqrt{2} z, z^2)$, that is, $X_0=\Span\{f_0\}:S^2\to \CP^2$ is the Veronese map $V_0^2$. Moreover, in this case, from \eqref{P(z)} one obtains   
	$P(z,\bar z)=0$, hence $\widehat Q=0$, which means, in view of Remark \ref{remark}, that $\psi$ is $S^1$-invariant, with $X_0=V_0^2$ and $X_1=V_0^2\oplus V_1^2$. Hence, taking $\lambda=-1$ and $r=2$ in \eqref{PSIinvariant},  we conclude that
	$\psi=\psi_{V_1^2}-\psi_{V_1^2}^\perp$ (up to unitary congruence).  
\end{proof}
Together with Lemma \ref{lem:tounram}, Theorem \ref{thm:unramified2} gives:

\begin{corollary}\label{thm:unramified3}
	If $\psi:S^2 \to \U(3)$ is a totally unramified minimal immersion with constant curvature and uniton number $2$, then $\psi$ is unitarily congruent to $\pi_{V_1^2}-\pi_{V_1^2}^\perp$. 
\end{corollary}

\subsection*{AI disclosure}The authors used ChatGPT (OpenAI) to improve the English language, identify typographical errors, and assist in checking mathematical arguments in a preliminary version of this paper. All mathematical ideas and proof strategies are due to the authors, who have independently verified the final content and take full responsibility for it.

\bibliographystyle{amsplain}
\bibliography{references}

\end{document}